\documentclass [11pt] {amsart}
\usepackage{bbm}
\usepackage{amssymb}
\usepackage{epic}
\usepackage{ecltree}
\usepackage{tikz}
\newtheorem{theorem}{Theorem}
\newtheorem{lemma}[theorem]{Lemma}
\newtheorem{proposition}[theorem]{Proposition}
\newtheorem{corollary}[theorem]{Corollary}

\newtheorem{question}[theorem]{Question}

\theoremstyle{definition}

\newtheorem{example}[theorem]{Examples}
\newtheorem{remark}[theorem]{Remark}

\begin{document}
\title{New characterizations of the unit vector basis of $c_0$ or $\ell_p$}


\author[P. G. Casazza]{Peter G. Casazza}

\address{Department of Mathematics, University of Missouri, Columbia, MO 65211-4100, USA}

\email{casazzap@missouri.edu}


\author[S. J. Dilworth]{Stephen J. Dilworth}

\address{Department of Mathematics, University of South Carolina, Columbia,
SC 29208, USA}

\email{dilworth@math.sc.edu}


\author[D. Kutzarova]{Denka Kutzarova}

\address{Department of Mathematics, University of Illinois Urbana-Champaign,
Urbana, IL 61807, USA; Institute of Mathematics and Informatics, Bulgarian Academy of Sciences, Sofia, Bulgaria}

\email{denka@illinois.edu}


\author[P. Motakis]{Pavlos Motakis}

\address{Department of Mathematics and Statistics, York University, 4700 Keele Street, Toronto, Ontario, M3J 1P3, Canada}

\email{pmotakis@yorku.ca}


\thanks{The first author was supported by
 NSF DMS 1906025.  The second author was supported by Simons Foundation Collaboration Grant No. 849142.
 The third author was supported by Simons Foundation Collaboration Grant No. 636954. The fourth author was supported by NSERC Grant RGPIN-2021-03639.}




\begin{abstract} Motivated by Altshuler's  famous characterization of  the unit vector basis of $c_0$ or  $\ell_p$  among  symmetric bases \cite{A1}, 
we obtain similar characterizations among democratic bases  and among  bidemocratic bases. We also prove a separate characterization of the unit vector basis of $\ell_1$.
\end{abstract}

\maketitle

\section{Introduction}

Let $(X,\|\cdot\|)$ be a Banach space with dual space $(X^*,\|\cdot\|_{*})$, let $(e_i)$ be a semi-normalized basic sequence in $X$, and let $\alpha = \sum_{i=1}^n a_i e_i$
and $\beta = \sum_{i=1}^m b_i e_i$,  where $b_m \ne 0$, belong to $\operatorname{span}((e_i))$. Define \begin{equation} \label{eq: multiplication}
\alpha \otimes \beta = \sum_{i=1}^n  a_i(\sum_{j=1}^m b_j e_{(i-1)m+j})\end{equation} and $\alpha \otimes 0 = 0$.
Note that   $(\operatorname{span}((e_i)),\otimes)$ is a noncommutative semigroup with identity element $e_1$. However, the semigroup multiplication is not continuous.

 If $(e_i)$ is equivalent to the unit vector basis  (u.v.b.) of  $c_0$ or $\ell_p$ ($1 \le p < \infty$) then
 there exists $K>0$ such that for all $\alpha, \beta \in \operatorname{span}((e_i))$,  \begin{equation} \label{eq: bothsides} \frac{1}{K}  \|\alpha\| \|\beta\| \le \|\alpha \otimes \beta \| \le K \|\alpha\| \|\beta\|.
 \end{equation} 
 
  The following is the  main open question related to this note.
  \begin{question} \label{question} Let $(e_i)$ be a semi-normalized basis for $X$ satisfying  \eqref{eq: bothsides}. Is $(e_i)$ equivalent to the u.v.b.\  of  $c_0$ or $\ell_p$ for some $1 \le p < \infty$?
  \end{question} We prove below that  Question~\ref{question} has a positive answer when  $(e_i)$  is either bidemocratic, almost greedy, or invariant under spreading.
  However, we do not know the answer  in general or  for other natural classes of bases such as the class of unconditional bases.
  
  Suppose that  $(e_i)$ is a symmetric basis for $X$ and that  both $X$ and $[(e^*_i)]$ have a unique symmetric basic sequence up to equivalence. 
  Altshuler \cite{A1} proved that $(e_i)$ is equivalent to the u.v.b.\ of $c_0$ or $\ell_p$. This theorem was recently extended to subsymmetric bases  \cite{CDKM}. In Section~\ref{sec: bidemocratic}  we adapt Altshuler's proof to give an answer to Question~\ref{question} for bidemocratic bases (see Section~\ref{sec: notation} for the definition of democratic and bidemocratic bases)  by imposing on  $(e_i)$ a  condition that is formally weaker  than  \eqref{eq: bothsides}. In Section~\ref{sec: p=1case}  we provide a new characterization of $\ell_1$ which implies a solution to Question~\ref{question}
  under additional assumptions. The proof uses  Szemer\'edi's theorem on arithmetic progressions.  In  Section~\ref{sec: democratic}  we use the latter characterization  to  provide an answer to Question~\ref{question}  for the class of democratic bases.
  
 The final section contains results about subsequences. We prove that Question~\ref{question} has a positive solution if $(e_i)$ is invariant under spreading or, more generally,  if every subsequence of $(e_i)$ satisfies \eqref{eq: bothsides}. Using infinite Ramsey theory we  also prove a related characterization of basic sequences which are saturated by subsequences equivalent to the u.v.b. of  $c_0$ or $\ell_p$.

\section{Notation and terminology} \label{sec: notation}

We use standard Banach space theory notation and terminology as in \cite{LT1}.  We also require some terminology and results from the theory of greedy bases
which we briefly review here. For further information on this topic we refer the reader to \cite{AK} and \cite{T}.

Let  $(e_i)$ be a semi-normalized basic sequence in $X$. For finite $A \subset \mathbb{N}$, let $\lambda(A) = \|\sum_{i \in A} e_i\|$. The \textit{ fundamental function} $(\Phi(n))$ of $(e_i)$ is defined by
$$\Phi(n) = \sup \{\lambda(A) \colon |A| \le n\}.$$
We say that $(e_i)$ is \textit{democratic} with constant $\Delta$ if  $\Phi(|A|) \le \Delta \lambda(A)$ for all finite $A \subset \mathbb{N}$. Democratic bases were introduced in \cite{KT1} in order to prove the celebrated characterization of  greedy bases as  unconditional and democratic.

 We say that $(e_i)$ is \textit{unconditional for constant coefficients} if there exists $C>0$ such that, for all finite $A \subset \mathbb{N}$ and all choices of signs $\pm$,
$$\frac{1}{C} \lambda(A) \le \|\sum_{i \in A} \pm e_i\|  \le C\ \lambda (A).$$

Now suppose that $(e_i)$ is a Schauder basis for $X$ with biorthogonal functionals $(e_i^*) \subset X^*$. Let $(\Phi^*(n))$ be the fundamental function of $(e_i^*)$.  We say that  $(e_i)$ is \textit{bidemocratic} if  $$\Phi^*(n) \asymp \frac{n}{\Phi(n)}.$$
(We write $a_n \asymp b_n$ if there exists $C>0$ such that  $a_n/C \le b_n \le Ca_n$ for all $n\in \mathbb{N}$.)
It is known that if $(e_i)$ is bidemocratic then both $(e_i)$ and $(e_i^*)$ are democratic and unconditional for constant
 coefficients \cite[Prop. 4.2]{DKKT}.
Moreover, every subsymmetric basis is bidemocratic \cite[Prop. 3.a.6]{LT1}.

\begin{remark} \label{rem: extremecases} As is well-known,  if $(e_i)$ is bidemocratic and $(\Phi(n))$ is bounded then there exists $C>0$ such that $\|\sum_{i-1}^n \pm e_i\| \le C$ for all $n\ge1$ and for all choices of signs. Hence $(e_i)$ is equivalent to the u.v.b.\ of $c_0$. 
On the other hand, if $(e_i)$ is bidemocratic and $\Phi(n) \asymp n$, then $(\Phi^*(n))$ is bounded. Hence $(e^*_i)$ is equivalent to the u.v.b.\ of $c_0$. By duality, $(e_i)$ is equivalent to the u.v.b.\ of $\ell_1$.
\end{remark}
\begin{example} \label{examples} 1)  For $1<p<\infty$, let $(e_i)$ be the u.v.b.\ of $p$-convexified Tsirelson space $T_p$ \cite{FJ}. Then $(e_i)$ is democratic and unconditional and $\Phi(n) \asymp n^{1/p}$. Since $(n^{1/p})$ has the \textit{upper regularity property} (see \cite[p.586]{DKKT})  it follows  from \cite[Prop.5.4]{DKKT} that $(e_i)$ is bidemocratic. Moreover, from the  special properties  of block bases in $T_p$  \cite[Cor.  II.5]{CS} it follows that  there exists $K>0$ such that for all $\alpha, \beta \in
 \operatorname{span}((e_i))$,
\begin{equation} \label{eq: multiplicative} \|\alpha\| \|\beta\| \le K \|\alpha \otimes \beta\|. \end{equation} 
2)  Let $(e_i)$ be a subsymmetric basis and  suppose that all subsymmetric block bases of $(e_i)$ are equivalent to $(e_i)$.
Altshuler \cite{A} constructed the first symmetric example  of this type (other than the u.v.b.\ of $c_0$ or $\ell_p$). A second symmetric example was constructed in \cite{CS}.  Recently, the first example of a subsymmetric basis
of this type,   which in addition is \textit{not} symmetric, was constructed in \cite{CDKM}. 

It was proved  in \cite[Lemma~17]{CDKM} that if $(e_i)$ is any  subsymmetric basis of this type, 
then there exists $K>0$ such that  for all $\alpha^*, \beta^* \in \operatorname{span}((e^*_i))$,
$$ \|\alpha^*\|_* \|\beta^*\|_* \le K \|\alpha^* \otimes \beta^*\|_*.$$

\begin{remark}
\label{condition implication}
Suppose that $(e_i)$ is a basis that satisfies \eqref{eq: bothsides} with constant $K$. Then, there exists a constant $\tilde K$ such that for every $n\in\mathbb{N}$ and $\alpha \in \operatorname{span}((e_i))$, $\alpha^* \in \operatorname{span}((e^*_i))$,
\begin{equation} \label{eq: alln}
\|\alpha\|^n \le \tilde K^n \|\alpha^n\|\quad\text{and}\quad \|\alpha^{*}\|_*^n \le \tilde K^n \|\alpha^{*n}\|_*.
\end{equation} 
Note that the left inequality follows simply by iterating \eqref{eq: bothsides}. For the right one, assume without loss of generality that $(e_i)$ is normalized and monotone. First, make the following easy observation. If $a = \sum_{i=1}^ma_ie_i$ and $a^* = \sum_{i=1}^mb_ie_i^*$ with $a_m\neq 0$ and $b_m\neq 0$ then, for all $n\in\mathbb{N}$, $a^{*n}(a^n) = (a^*(a))^n$.

Next, take an $a^* = \sum_{i=1}^mb_ie_i^*$ and $n\in\mathbb{N}$. By monotonicity, find a norm-one $a = \sum_{i=1}^\ell a_ie_i$ with $\ell\leq m$ such that $a^*(a) = \|a^*\|_*$. Assume first that $\ell < m$. For $\varepsilon > 0$ let $a_\varepsilon = a + \varepsilon e_m$ and note that $|a^*(a_\varepsilon)| \geq (1-\varepsilon)\|a^*\|_*$, while $\|a_\varepsilon\| \leq 1 + \varepsilon$. Therefore,
\[(1-\varepsilon)^{n}\|a^*\|_*^n \leq |a^{*}(a_\varepsilon)|^n = |a^{*n}(a_\varepsilon^n)| \leq \|a^{*n}\|_*\|a_\varepsilon^n\|_* \leq K^n(1+\varepsilon)^n\|a^{*n}\|\]
and let $\varepsilon\to0$. If $\ell =m$ the argument is slightly simpler and does not require the perturbation of $a$.  
\end{remark}

 \section{Bidemocratic bases} \label{sec: bidemocratic}

The main theorem of this section is the following.
\begin{theorem}\label{main thm: bidemocratic}
Suppose that $(e_i)$ is a bidemocratic basis for $X$ that satisfies  \eqref{eq: bothsides}. Then $(e_i)$ is equivalent to the u.v.b.\ of $c_0$ or $\ell_p$.
\end{theorem}
The above is an immediate consequence of the next proposition and Remark \ref{condition implication}.

\begin{proposition} \label{thm: bidemocratic}  Suppose that $(e_i)$ is a bidemocratic basis for $X$ that satisfies \eqref{eq: alln}. Then $(e_i)$ is equivalent to the u.v.b.\ of $c_0$ or $\ell_p$. \end{proposition}

\begin{proof}  For $n \ge 1$,  let $\lambda(n) = \|\sum_{i=1}^n e_i\|$ and $\lambda^*(n) =\|\sum_{i=1}^n e^*_i\|_*$. Note that if $\alpha = \sum_{i=1}^n e_i$ and $k \ge 1$, then $\alpha^k = \sum_{i=1}^{n^k}e_i$. So \eqref{eq: alln}  applied to $
\alpha$ yields \begin{equation}\label{eq: upperest} [\lambda(n)]^k \le K^k \lambda(n^k). \end{equation}
Similarly  \eqref{eq: alln} applied to $\alpha^* = \sum_{i=1}^n e_i^*$ gives
\begin{equation*}  [\lambda^*(n)]^k \le K^k \lambda^*(n^k). \end{equation*}
Since $(e_i)$ is bidemocratic, there exists $C>0$ such that 
\begin{equation*} \frac{n}{\lambda(n)} \le \lambda^*(n) \le C \frac{n}{\lambda(n)}. \end{equation*}
Hence $$ \left[\frac{n}{\lambda(n)}\right]^k \le CK^k \frac{n^k}{\lambda(n^k)},$$
i.e., \begin{equation} \label{eq: lowerest} \lambda(n^k) \le CK^k[\lambda(n)]^k. \end{equation}
By the proof  (see p.60) of  \cite[Th. 2.a.9]{LT1}, \eqref{eq: upperest} and \eqref{eq: lowerest} imply that  $\lambda(n) \asymp n^{1/p}$ for some $1 \le p \le \infty$.
Since $(e_i)$ is bidemocratic, it follows that  $\Phi(n) \asymp n^{1/p}$ and $\Phi^*(n) \asymp n^{1/q}$, where $q=p/(p-1)$. By Remark~\ref{rem: extremecases}, if $p=1$ or $p=\infty$ then $(e_i)$ is equivalent to the u.v.b.\ of $\ell_1$ or $c_0$ respectively. 

So suppose $1<p<\infty$. Consider $\alpha = \sum_{i=1}^m a_i e_i \in \operatorname{span}((e_i))$.  By \eqref{eq: alln}, $\|\alpha\|^n \le K^n \|\alpha^n\|$
for each $n \ge 1$. Note that
$$\alpha^n = \sum_{i_1+\dots +i_m =n} a_1^{i_1}a_2^{i_2}\cdots a_m^{i_m}(\sum_{j \in A(i_1,\dots,i_m)} e_j),$$
where $$ |A(i_1,\dots,i_m)| = \binom{n}{i_1 \dots i_m}.$$ Hence  \begin{align*} \|\alpha^n\| &\le 
 \sum_{i_1+\dots+ i_m =n} |a_1^{i_1}a_2^{i_2}\cdots a_m^{i_m}| \lambda(|A(i_1,\dots,i_m)|)\\
 & \le  \sum_{i_1+\dots+ i_m =n} |a_1^{i_1}a_2^{i_2}\cdots a_m^{i_m}| \Phi(|A(i_1,\dots,i_m)|)\\&\le
 C \sum_{i_1+\dots+ i_m =n} |a_1^{i_1}a_2^{i_2}\cdots a_m^{i_m}|  \binom{n}{i_1 \dots i_m}^{1/p}
 \intertext{(for some $C>0$)}
 &\le C(n+1)^{m/q}(\sum_{i_1+\dots+ i_m =n} |a_1^{i_1}a_2^{i_2}\cdots a_m^{i_m}|^p  \binom{n}{i_1 \dots i_m})^{1/p}
\end{align*} by H\"{o}lder's inequality. Hence
$$\|\alpha\|^n \le K^n C (n+1)^{m/q} (\sum_{i=1}^m |a_i|^p)^{n/p}.$$
Taking the $n^{th}$ root and then the  limit as $n \rightarrow \infty$ gives
\begin{equation} \label{eq: upper}\|\sum_{i=1}^m a_i e_i\| \le K (\sum_{i=1}^m |a_i|^p)^{1/p}.\end{equation} 
Since $(e_i^*)$  also satisfies \eqref{eq: alln} and $\Phi^*(n) \asymp n^{1/q}$, the same argument gives
\begin{equation} \label{eq: lower} \|\sum_{i=1}^m a_i e^*_i\|_* \le K (\sum_{i=1}^m |a_i|^q)^{1/q}. \end{equation}
Hence, by duality, $(e_i)$ is equivalent to the u.v.b.\ of $\ell_p$.
\end{proof}
The following corollary replaces \eqref{eq: alln}  by a weaker condition but adds an assumption about the fundamental function.
\begin{corollary} Let $1 < p < \infty$. Suppose that $(e_i)$ is a bidemocratic basis for $X$ such that $\Phi(n) \asymp n^{1/p}$. Suppose also that there exists $K>0$  such that
for all $\alpha \in \operatorname{span}((e_i))$ and $\alpha^* \in \operatorname{span}((e^*_i))$,
\begin{equation} \label{eq: squares}  \|\alpha\|^2 \le K \|\alpha^2\|\quad\text{and}\quad \|\alpha^{*}\|_*^2 \le K \|\alpha^{*2}\|_*.
\end{equation} 
Then $(e_i)$ is equivalent to the u.v.b.\ of $\ell_p$.
\end{corollary}
\begin{proof} Iteration of  \eqref{eq: squares}  yields  $\|\alpha\|^n \le K^n \|\alpha^n\|$ and   $\|\alpha^{*}\|_*^n  \le K^n \|\alpha^{*n}\|_*$
when $n$ is a power of $2$. This is enough to prove   \eqref{eq: upper} and \eqref{eq: lower}  and conclude the proof as above.
\end{proof}
\section{A characterization of $\ell_1$} \label{sec: p=1case}\begin{theorem} \label{thm: p=1case}
 Let $(e_i)$ be a semi-normalized basis for $X$. Suppose that there exists $K>0$ such that, for all $\alpha,\beta \in \operatorname{span}((e_i))$, 
\begin{equation} \label{eq: upperineq}\|\alpha \otimes \beta\| \le K \|\alpha\| \|\beta\|,\end{equation}
and also that \begin{equation} \label{eq: averageoversigns} \limsup \frac{1}{n}\underset{\pm}{\operatorname{Ave}}\left\|\sum_{i=1}^n \pm e_i\right\|>0. \end{equation}
Then $(e_i)$ is equivalent to the u.v.b.\ of $\ell_1$.
\end{theorem}
\begin{remark} Note that \eqref{eq: bothsides} imples that $\lambda(n) \asymp n^{1/p}$ for some $1 \le p \le \infty$ (see \eqref{eq: twosidedinequality} below). If $p=1$ and $(e_i)$ is unconditional for constant coefficients then \eqref{eq: averageoversigns} is satisfied. So Theorem~\ref{thm: p=1case} gives a positive answer to Question~\ref{question}
in this case.
\end{remark}
 \begin{proof}
We may assume that $\|e_i\| \le 1$ for all $i \ge 1$. By assumption there exist an infinite $M \subseteq \mathbb{N}$ and $\delta>0$ such that 
$$\underset{\pm}{\operatorname{Ave}}\left\|\sum_{i=1}^m \pm e_i\right\|>\delta m$$ for all $m \in M$. By Elton's `$\ell_1^n$ theorem' \cite{E}
there exist $\delta_1>0$ and $c>0$ such that for each $m \in M$ there exists $A_m \subset \{1,2,\dots,m\}$, with $|A_m| \ge \delta_1 m$,
such that for all scalars $(a_i)_{i \in A_m}$,
$$c \sum_{i \in A_m} |a_i| \le \|\sum_{i \in A_m} a_i e_i\| \le \sum_{i \in A_m} |a_i|.$$ Let $k \in \mathbb{N}$. By Szemer\'edi's theorem \cite{S},   when $m$ is sufficiently large $A_m$ contains an arithmetic progression of length $k$, $\{n_1, n_1+d, n_1+2d,\dots, n_1+(k-1)d\}$.
Let $n_1 = bd + r$, where $1 \le r \le d$. Fix scalars $(a_i)_{i=1}^k$ and $\varepsilon>0$,  and set
$$\alpha = \sum_{i=1}^k a_i e_{b+i}\quad\text{and}\quad \beta_\varepsilon=e_r + \varepsilon \sum_{i=r+1}^d e_i.$$
Note that
$$\alpha\otimes \beta _\varepsilon= \sum_{i=1}^k a_i e_{n_1+(i-1)d} +   \varepsilon y$$
for some $y \in \operatorname{span}((e_i))$. Thus, applying \eqref{eq: upperineq},
\begin{align*} c\sum_{i=1}^k |a_i| - \varepsilon \|y\| \le \|\alpha \otimes \beta_\varepsilon\| \le K \|\alpha\| \|\beta_\varepsilon\|
 \le K(1 + (d-r)\varepsilon)\|\alpha\|.
\end{align*} Taking the limit as $\varepsilon \rightarrow 0$, we get $$\|\alpha\| \ge \frac{c}{K} \sum_{i=1}^k |a_i|.$$ Let $n = \lfloor k/3 \rfloor$.
There exists $s \in \mathbb{N}$ such that $[(s-1)n +1,sn] \subset [b+1,b+k]$. We may assume that $a_n \ne 0$. Then, applying \eqref{eq: upperineq} again,
$$ \|\sum_{i=1}^n a_ie_i\| \ge \frac{1}{K} \|e_s \otimes \sum_{i=1}^n a_i e_i\|=\frac{1}{K}\|\sum_{i=1}^n a_i e_{(s-1)n+i}\| \ge \frac{c}{K^2}\sum_{i=1}^n |a_i|.$$
Since $k$ (and hence $n$) are arbitrary, it follows that $(e_i)$ is equivalent to the u.v.b.\ of $\ell_1$.
\end{proof} \begin{corollary} \label{cor: characterization} Suppose that $(e_i)$   satisfies \eqref{eq: upperineq} and  is unconditional for constant coefficients and that 
$\lambda(n) \asymp n$. Then $(e_i)$ is equivalent to the u.v.b.\ of $\ell_1$. \end{corollary} 
\begin{remark}  Corollary~\ref{cor: characterization} does not admit an $\ell_p$ version  for $p>1$. To see this,  let $q = p/(p-1)$. As observed  in Examples~\ref{examples} the u.v.b.\ $(e_i)$ of $T_q^*$ is unconditional and  bidemocratic with $\phi(n) \asymp n^{1/p}$. Moreover,  \eqref{eq: multiplicative} dualizes  due to the special property of block bases of  $T_q$ \cite[Cor. II.5]{CS}.
In particular,  $(e_i)$ satisfies  \eqref{eq: upperineq}. Thus, Corollary~\ref{cor: characterization} does not admit an $\ell_p$ version  for $p>1$.
\end{remark} \section{Democratic bases} \label{sec: democratic}
Democratic bases are more general than bidemocratic bases. In this section we prove a characterization of the u.v.b.\ of $\ell_p$ or $c_0$ among democratic bases. 

The characterization involves the notion of quasi-greedy basis which we now recall.
Let $(e_i)$ be a basis for $X$. For $x \in X$ and $\delta>0$, let 
$$G(x,\delta) = \{i \in \mathbb{N} \colon |e_i^*(x)| \ge \delta\}\quad\text{and}\quad \mathcal{G}_\delta(x)= \sum_{i \in G(x,\delta)} e_i^*(x) e_i.$$
Then $(e_i)$ is \textit{quasi-greedy} if there exists $A>0$ such that for all $x \in X$ and for all $\delta>0$, $\|\mathcal{G}_\delta(x)\| \le A \|x\|$.
Quasi-greedy bases were introduced in \cite {KT1} in connection with the Thresholding Greedy Algorithm.
While unconditional bases are quasi-greedy there are also important   examples of  conditional quasi-greedy bases \cite{AADK, KT1, W}. However, quasi-greedy bases are always unconditional for constant coefficients \cite{W}. We refer to \cite{DKKT} for the definition of an \textit{almost greedy} basis. Almost greediness of a basis  is proved    in \cite{DKKT} to be equivalent to being   quasi-greedy and democratic.
\begin{theorem}
\label{almost greedy}
Let $(e_i)$ be a quasi-greedy democratic (i.e., almost greedy) basis for $X$ which satisfies  \eqref{eq: bothsides}. Then $(e_i)$ is equivalent to the u.v.b.\  of $c_0$ or $\ell_p$.
\end{theorem}
\begin{proof}  By considering $\alpha= \sum_{i=1}^m e_i$ and $\beta = \sum_{i=1}^n e_i$, \eqref{eq: bothsides} implies that, for all
 $m,n \in \mathbb{N}$, \begin{equation} \label{eq: twosidedinequality} \frac{1}{K} \lambda(m) \lambda(n) \le \lambda(mn) \le K \lambda(m) \lambda(n).
 \end{equation}
 Hence by \cite{LT1}, $\lambda(n) \asymp n^{1/p}$ for some $1 \le p \le \infty$. Since $(e_i)$ is democratic, it follows that $\Phi(n) \asymp n^{1/p}$.
 
 If $(\Phi(n))$ is bounded, then, since $(e_i)$ is unconditional for constant coefficients,  there exists $C>0$ such that $\|\sum_{i=1}^n \pm e_i\| \le C$ for all $n \ge 1$ and all choices of signs. Hence $(e_i)$ is equivalent to the u.v.b.\ of $c_0$. 
 
 If $\phi(n) \asymp n$ then  by the democratic assumption and unconditionality for constant coefficients, it follows that $$\underset{\pm}{\operatorname{Ave}}\left\|\sum_{i=1}^n \pm e_i\right\| \asymp n.$$
 Hence by Theorem~\ref{thm: p=1case}, $(e_i)$ is equivalent to the u.v.b.\  of $\ell_1$. 
 
 So suppose that $1<p<\infty$. The sequence $(n^{1/p})$ has the `upper regularity property' for $1<p<\infty$ (see \cite[p.\ 586]{DKKT}).
 So $(e_i)$ is quasi-greedy and democratic  and $(\Phi(n))$ has the upper regularity property. Hence by \cite[Prop. 4.4]{DKKT}, $(e_i)$ is bidemocratic.
 In particular, $\Phi^*(n) \asymp n^{1/q}$, where $q=p/(p-1)$.
 
 The proof is now concluded as in Proposition~\ref{thm: bidemocratic} above. More precisely, the  left-hand   inequality of \eqref{eq: bothsides}
 gives an upper $p$-estimate for $(e_i)$ and  \eqref{condition implication}  gives an upper $q$-estimate for $(e_i^*)$.
\end{proof}
\section{Subsequences}
In this section we use results from the theory of spreading models initiated by Brunel and Sucheston \cite{BS}.  We start with a brief review.  A basic sequence $(e_i)$ is \textit{invariant under spreading} (IS)  if there exists  $C>0$ such that 
$$\frac{1}{C} \|\sum_{i=1}^n a_i e_i\| \le \| \sum _{j=1}^n a_j e_{i_j}\| \le C \|\sum_{i=1}^n a_i e_i\|$$
for all $n \ge 1$,  $i_1<\cdots<i_n$, and scalars $(a_i)_{i=1}^n$. 

We say $(e_i)$ \textit{generates a spreading model}  if there exists  a  Banach space $(Y,\|\cdot\|_Y)$ with basis $(s_i)$ such that
$$\|\sum_{i=1}^n a_i s_i\|_Y = \underset{i_1 <\cdots<i_n}{\lim_{i_1 \rightarrow \infty}}\|\sum_{j=1}^n a_j e_{i_j}\|,$$
 for all $n\ge 1$ and  scalars $(a_i)_{i=1}^n$.   It was proved in \cite{BS} that every basic sequence has a subsequence $(x_i)$ which generates a spreading model. Note that $(s_i)$ is IS with $C=1$.  Moreover, $(s_{2i}-s_{2i-1})$ is suppression $1$-unconditional \cite{BS}.
Note that an IS sequence $(e_i)$ is equivalent to the spreading model $(s_i)$ generated by a subsequence of $(e_i)$. Hence $(e_{2i}-e_{2i-1})$ is IS and unconditional. 

 \begin{theorem} \label{thm: IS} Suppose that $(e_i)$ is IS and satisfies \eqref{eq: bothsides}. Then $(e_i)$ is equivalent to the u.v.b. of $c_0$ or  $\ell_p$.
\end{theorem} 
\begin{proof} As remarked above, $(e_{2i}-e_{2i-1})$ is IS and  unconditional. Because, for every choice of scalars $a_1,\ldots,a_n$,
\[\Big(\sum_{i=1}^na_ie_i\Big)\otimes(e_2-e_1) = \sum_{i=1}^na_i(e_{2i}-e_{2i-1}),\]
\eqref{eq: bothsides}, yields that $(e_i)_i$ is equivalent to $(e_{2i}-e_{2i-1})$. Therefore, $(e_i)_i$ is almost greedy and thus by Theorem \ref{almost greedy} the result follows.
\end{proof}

To state the next result we first make some clarifying remarks to avoid a possible source of confusion. For any  given semi-normalized  basic sequence $(e_i)$ we defined a multiplication $\otimes$ on $\operatorname{span}((e_i))$ by \eqref{eq: multiplication}. We must emphasize, however, that for a different choice of  basic sequence, $(f_i)$ say, the corresponding multiplication will also be different. But to avoid cumbersome notation we use the same symbol $\otimes$ for both.
Likewise, when we say in the next result that all subsequences of $(e_i)$   satisfy \eqref{eq: bothsides}, it is to be understood that the constant $K$ in \eqref{eq: bothsides} 
will depend on the  subsequence, i.e., we are not assuming a priori that there is a uniform $K$  for all subsequences. (Of course, as the result shows,  a uniform $K$ does in fact  exist.)
\begin{theorem} \label{thm: hereditary}  Let $(e_i)$ be a semi-normalized basic sequence. Suppose that every subsequence of $(e_i)$ satisfies \eqref{eq: bothsides}. 
Then $(e_i)$ is equivalent to the u.v.b. of $c_0$ or $\ell_p$.
\end{theorem} The proof requires the following lemma. \begin{lemma} \label{lem: uniformlyequivalent}
 Let $(e_i)$ be a semi-normalized basic sequence which satisfies \eqref{eq: bothsides}. Then all sequences of the form $(e_{mn+i})_{i=1}^n$
 ($m,n \in \mathbb{N}$) are uniformly equivalent to $(e_i)_{i=1}^n$, i.e., there exists $C>0$ such that for all $m,n \in \mathbb{N}$ and all scalars $(a_i)_{i=1}^n$,
 $$\frac{1}{C}\|\sum_{i=1}^n a_i e_i\| \le \|\sum_{i=1}^n a_i e_{mn+i}\| \le C \|\sum_{i=1}^n a_i e_i\|.$$
\end{lemma} \begin{proof} This follows at once from \eqref{eq: bothsides} along with the observation that
$$ e_{m+1} \otimes (\sum_{i=1}^n a_i e_i )= \sum_{i=1}^n a_i e_{mn+i}.$$
\end{proof} \begin{proof}[Proof of Theorem~\ref{thm: hereditary}] 
 By Theorem~\ref{thm: IS}, it suffices to prove that $(e_i)$ is IS. Let $(y_i)$ be a subsequence of $(e_i)$ which generates a spreading model $(Y, \|\cdot\|_Y)$ with basis $(s_i)$. We define a subsequence $(x_i)$ of $(e_i)$ inductively. For $1 \le i \le 3$, let $x_i = e_i$.  For the inductive step, suppose that $n \ge 1$ and that 
 $(x_i)_{i=1}^{3^n}$ have been defined with $x_i = e_{N(i)}$, where $(N(i))_{i=1}^{3^n}$ is strictly increasing. Choose $m \in \mathbb{N}$ with $m3^n>N(3^n)$ and define $x_{3^n + i} = e_{m3^n +i}$ for $1 \le i \le 3^n$. Thus, $x_i$ has now  been defined for $1 \le i \le 2\cdot 3^n$.

 Now choose $p > (m+1)3^n$ such that \begin{equation} \label{eq: equivspreadingmodel} \frac{1}{2}\|\sum_{i=1}^{3^n} a_i s_i\|_Y \le \|\sum_{i=1}^{3^n} a_i y_{p+i}\|
 \le 2 \|\sum_{i=1}^{3^n} a_i s_i\|_Y
 \end{equation} for all scalars $(a_i)_{i=1}^{3^n}$. This is possible because $(y_i)$ generates the spreading model with basis $(s_i)$. Define
 $x_{2\cdot 3^n+i} = y_{p+i}$ for $1 \le i \le 3^n$. This completes the inductive definition of $x_i = e_{N(i)}$ for $1 \le i \le 3^{n+1}$. Note that
 $$N(3^n+1) = m3^n+1>N(3^n)$$ and
 $$N(2\cdot 3^n + 1) \ge p+1) > (m+1)\cdot 3^n =  N(2 \cdot 3^n).$$ Hence $(N(i))_{i=1}^{3^{n+1}}$ is strictly increasing as desired.
 
  By assumption, $(x_i)$ satisfies \eqref{eq: bothsides} for some $K>0$. Hence by Lemma~\ref{lem: uniformlyequivalent}, $(x_{3^n+i})_{i=1}^{3^n}
  = (e_{m\cdot3^n+i})_{i=1}^{3^n}$ is uniformly equivalent to $(e_i)_{i=1}^{3^n}$. Again by Lemma~\ref{lem: uniformlyequivalent},
  $(x_{3^n+i})_{i=1}^{3^n}$  is uniformly equivalent to $(x_{2\cdot3^n+i})_{i=1}^{3^n}$, which in turn is uniformly equivalent to $(s_i)_{i=1}^{3^n}$ by \eqref{eq: equivspreadingmodel}. So $(e_i)_{i=1}^{3^n}$ is uniformly equivalent to $(s_i)_{i=1}^{3^n}$, i.e., $(e_i)$ is equivalent to $(s_i)$ as desired.
\end{proof}
We conclude with a characterization of basic sequences that are saturated by subsequences equivalent to the u.v.b. of $c_0$ or $\ell_p$.

 Let $[\mathbb{N}]^\omega$ denote the collection of increasing sequences  $(n_k)_{k=1}^\infty$ of natual numbers endowed with  the product topology.
 \begin{theorem} Let $(e_i)$ be a semi-normalized basic sequence. The following are equivalent: \begin{itemize}
\item[(a)] every subsequence of $(e_i)$ contains a further subsequence equivalent to the u.v.b. of $c_0$ or $\ell_p$;
\item[(b)]  every subsequence of $(e_i)$ contains a further subsequence satisfying \eqref{eq: bothsides} for some $K>0$. \end{itemize} \end{theorem}
\begin{proof} (a) $\Rightarrow$ (b) is obvious. Suppose (b) holds. Let $(f_i)$ be any subsequence of $(e_i)$. Let
$$B = \{ (n_k)_{k=1}^\infty \in [\mathbb{N}]^\omega \colon (f_{n_k})_{k=1}^\infty \text{\, satisfies \eqref{eq: bothsides} for some $K>0$}\}.$$
Then $B$ is easily seen to be a Borel set. Hence by the infinite Ramsey theorem of Galvin and Prikry \cite{GP} there  exists $(n_k)_{k=1}^\infty \in [\mathbb{N}]^\omega$
such that either every subsequence of $(n_k)_{k=1}^\infty$ belongs to $B$ or every subsequence belongs to the complement of $B$. The latter contradicts (b). By Theorem~\ref{thm: hereditary}, the former implies that $(f_{n_k})$ is equivalent to the u.v.b. of $c_0$ or $\ell_p$.
\end{proof}
\end{example}

\end{document}